\documentclass[12pt]{article}
\usepackage[margin=2.5cm]{geometry}
\usepackage[pdftex]{graphicx}
\usepackage{currvita}
\usepackage{hyperref}
\usepackage{amssymb}
\usepackage{amsthm}
\usepackage{amsfonts}
\usepackage{amsmath}
\usepackage{parskip}
\usepackage{array}
\usepackage{authblk}
\usepackage{enumerate}
\usepackage{enumitem}
\usepackage{titlesec}

\titlespacing\section{0pt}{12pt plus 4pt minus 2pt}{0pt plus 2pt minus 2pt}

\newtheoremstyle{exampstyle}
  {10pt} 
  {\topsep} 
  {} 
  {} 
  {\bfseries} 
  {.} 
  {.5em} 
  {} 

\theoremstyle{exampstyle} \newtheorem{thm}{Theorem}[section]
\newtheorem{lem}[thm]{Lemma}
\newtheorem{prop}[thm]{Proposition}
\newtheorem{cor}[thm]{Corollary}

\begin{document}
\title{An improved bound for disjoint directed cycles} 
\author{Matija Buci\'c}
\affil{University of Cambridge}
\date{}
\maketitle

\begin{abstract}
We show that every directed graph with minimum out-degree at least $18k$ contains at least $k$ vertex disjoint cycles. This is an improvement over the result of Alon who showed this result for digraphs of minimum out-degree at least $64k$. The main benefit of the argument is that getting better results for small values of $k$ allows for further improvements to the constant. 
\end{abstract}		
				
\section {Introduction}
In this paper all digraphs are considered to contain no parallel edges, but loops and bidirectional edges (pairs of edges which join two vertices in opposite directions) are allowed, but considered as cycles of length $1$ and $2$ respectively. Throughout this paper by a cycle we mean a directed cycle.

We define $f(k)$ to be the minimal integer $d$ for which any digraph $G,$ with all vertices having at least $d$ outgoing edges, contains $k$ pairwise vertex disjoint cycles as subgraphs, if such a $d$ does not exist we set $f(k)=\infty.$

The study of $f(k)$ was initiated by Bermond and Thomassen \cite{Ber_Tho81} who observed that as the complete directed graph\footnote{A complete digraph has no loops, and between any two vertices, contains a bidirectional edge.} on $2k-1$ vertices has out-degree $2k-2$ and contains at most $k-1$ disjoint cycles, it follows that $f(k) \ge 2k-1.$ They also conjecture that this is in fact tight, so that $f(k)=2k-1$. 

There has been plenty of work around this conjecture. Thomassen \cite{Tho83} showed that $f(2)=3$ and that $f(k)$ is always finite. In particular, he showed $f(k) \le (k+1)!$. Almost 15 years later Alon \cite{Alon96} improved significantly on this bound and showed that $f(k) \le 64k$. More recently P\'or and Sereni \cite{k=3} showed that $f(3)=5.$ The conjecture has received attention even in the case when the digraph in question is restricted to be a tournament, with Bessy,  Lichiardopol and Sereni \cite{tournaments1} showing that it holds for tournaments with bounded minimum in-degree and later completely resolved, in case of tournaments, by Bang-Jensen, Bessy and Thomass\'e \cite{tournaments2}

We improve on this bound to show $f(k) \le 18k$. Our proof follows closely that of Alon \cite{Alon96}, our main improvement is based on exploiting the fact that $f(3)=5$, due to P\'or and Sereni \cite{k=3}, which allows for significantly better bounds in a part of the argument.

\subsection{Definitions and notation}
For a vertex $v$ of an undirected graph $G$ we denote by $d_G(v)$ its degree and by $N_G(v)$ the set of its neighbours.

Given a digraph $G(V,E),$ for $x,y \in V$ we denote by $xy$ the edge from $x$ to $y$. If edge $xy$ exists we say $x$ is a \emph{parent} of $y$ and $y$ is a \emph{child} of $x$. If both edges $xy$ and $yx$ exist we say there is a \emph{bidirectional} edge joining $x$ and $y$. 

We define the \emph{out-degree} of a vertex $v$ of a digraph $G$ as the number of children of $v$ in $G$ and denote it as $d^+_G(v)$. We denote by $N_G^{+}(v)$ the set of children of $v$ in $G$. We define the \textit{in-degree} of a vertex $v$ as the number of parents of $v$ in $G$ and denote it by $d^-_G(v)$ and we denote the set of these parents as $N_G^{-}(v).$ The index digraph $G$ is omitted whenever there is no chance of confusion in regards to which graph is being referred to. 

We say a digraph $G$ is \textit{$k$-out} if every vertex in $G$ has out-degree at least $k$. A digraph is said to be \textit{exactly $k$-out} if each out-degree in the graph equals $k.$ \textit{$k$-in} and \textit{exactly $k$-in} graphs are defined analogously. 

We denote by $V(G)$ the set of vertices of a (directed) graph $G$, for any $X \subseteq V(G)$ we denote by $G[X]$ the subgraph of $G$ induced by $X$.

\section{Preliminary results}

We start by proving $f(2)=3$, first proved by Thomassen \cite{Tho83}. We present a similar proof of this result in order to illustrate some of the ideas, which we will reuse later in the proof of our main result.

\begin{lem} \label{lem:one_one}
$f(2)=3$
\end{lem}

\begin{proof} A complete digraph on $3$ vertices does not contain $2$ disjoint cycles, implying $f(2)\ge 3$. 

We now prove that any $3$-out digraph contains $2$ disjoint cycles. We proceed by induction on $n$, the number of vertices. For the base case we consider $n=4$ where the only possible digraph is the complete digraph which contains $2$ disjoint cycles of length $2$. We now assume that any $3$-out digraph, with $n-1 \ge 4$ vertices, contains $2$ disjoint cycles. 

Assuming there is a digraph on $n$ vertices failing our assumption, we can remove some of its edges to make it exactly $3$-out, such new digraph still does not contain $2$ disjoint cycles. We call this digraph $G$.

$G$ has no bidirectional edges. If $uw$ is a bidirectional edge then $G-\{u,w\}$ is still $1$-out so it contains a cycle, which paired with the $2$-cycle made by the bidirectional edge $uw$ gives the desired disjoint cycles.

The main idea allowing us to prove the result is using edge contractions. If there is an edge $uv$ such that $u,v$ have no common parent, we can modify $G$ to a new digraph $G'$ by removing $u$ and $v$ and adding a vertex $w$ having $N_{G'}^{+}(w)\equiv N_{G}^{+}(v)$ and $N_{G'}^{-}(w)\equiv N_{G}^{-}(u) \cup N_{G}^{-}(v)$. $G'$ has $n-1$ nodes and is still $3$-out, as $u$,$v$ had no common parent, so by the inductive assumption $G'$ contains $2$ disjoint cycles. If $w$ is not in any of the cycles they were contained in $G$ to start with and we are done. The other option is that $w$ is contained in one of the cycles, this implies the other cycle is in $G$. We distinguish the cases when the cycle in-edge of $w$ comes from $u$'s in-edge or from $v$'s in-edge. Replacing $w$ in the first case by $uv$ and in the second by $v$ we get the other cycle in $G$ and we are done.

The only remaining option is when every edge of $G$ has a \textit{witness} (defined as common parent to both its end-vertices).

Let $v$ be a vertex having the smallest in-degree. As there are $3n$ edges in total we have $d^-(v) \le 3.$

\begin{itemize}
	\item Case $1$: $d^-(v)=0$
	Then $G-\{v\}$ is $3$-out and has $n-1$ vertices so inductively we can find disjoint cycles which are also contained in $G$.

\item Case $2$: $d^-(v)=1$
	The edge ending in $v$ has no witnesses.

\item Case $3$: $d^-(v)=2$
	Let $N^{-}(v)\equiv \{u,w\}$. $u$ has to be the witness to $wv$ and $w$ to $uv$ implying $uw$ is a bidirectional edge which is a contradiction.
	
\item Case $4$: $d^-(v)=3$
	By our choice of $v,$ for all $x\in V(G)$ we have $d^-(x)\ge d^-(v)=3$ and $3n=\sum_{x \in V(G)} d^-(x) \ge nd^-(v)=3n$ so we need to have equality in $d^-(x)\ge d^-(v)$ for all $x$ implying all vertices have in-degree $3$.

	Given a vertex $x$ and its parents $u,v,w$ we show that $u,v,w$ make a $3$-cycle. Indeed, we notice that each of the witnesses of $ux,vx,wx$ must be among $u,v,w$ implying that the $u,v,w$ induced subgraph is $1$-in. This combined with the fact bidirectional edges do not exist, we conclude that $u,v,w$ make a $3$-cycle as claimed.
	
	If we reverse the edges of $G$ we notice that all the above arguments still apply, as $G$ is both exactly $3$-out and exactly $3$-in, so children of $x$ also form a $3$-cycle. As there are no bidirectional edges, the children and parents of $x$ are disjoint and give $2$ disjoint $3$-cycles.
\end{itemize}

\vspace{-0.77 cm}

\end{proof}

\section{The main result}

As noted before considering the complete digraph on $2k-1$ vertices we have $f(k) \ge 2k-1$. 


We start by defining a class of, in some sense, minimal counterexamples to the Bermond-Thomassen conjecture. The main reason for this is to fix a minor flaw in the way the argument of Alon in \cite{Alon96} is written, where he looks at a minimal counterexample to $f(k)\le 64k$ and shows Proposition \ref{critical} for it, but then also uses this proposition for graphs which are minimal counterexamples to different inequalities and as a consequence omits a rather easy case, specifically the second case of Corollary. \ref{cor:recursion}.

We now define digraphs which act as  For positive integers $r,k$ we say a digraph $G$ is \emph{$(k,r)$-critical} if the following properties hold: 
\begin{enumerate}[label=\arabic*)] 
\item $G$ is $r$-out, 
\item $G$ does not contain $k+1$ disjoint cycles,
\item Subject to properties $1)$ and $2)$ $G$ has the smallest number of vertices and
\item Subject to properties $1),2)$ and $3)$ $G$ has the smallest possible number of edges.
\item Any $(r-2)$-out digraph contains $k$ disjoint cycles.
\end{enumerate}
Note that existence of critical graphs for any $r,k$ is not a given, in fact such a graph exists if and only if the Bermond-Thomassen conjecture fails.

As the complete digraph on $r-1$ vertices is $(r-2)$-out and contains at most $(r-1)/2$ disjoint cycles, it follows by Property $5)$ that if a $(k,r)$-critical digraph exists we have $r-1 \ge 2k$.

A $(k,r)$-critical digraph $G$ can not have a bidirectional edge, as else it would contain $k+1$ disjoint cycles, one made by the bidirectional edge and $k$ given by the assumption on $r,k$ for the $(r-2)$-out digraph $G$ with the vertices of the bidirectional edge removed and this would contradict Property $2)$. Similarly a $(k,r)$-critical digraphs can not contain any loops.

\begin{prop}\label{prop:in-nbor}
Let $G(V,E)$ be a $(k,r)$-critical digraph. For any $v\in V$, $G[N^{-}(v)]$ is a $1$-out digraph.
\end{prop}

\begin{proof}
For every edge, its endpoints need to have a witness, as else we can contract along this edge which would contradict Property $3$ of $G$ being critical.

Given any vertex $u\in N^{-}(v)$, if the edge $uv$ has a witness $w$ we must have $w \in N^{-}(v)$ and $w$ is a parent of $u$ so $G[N^{-}(v)]$ is $1$-in and consequently contains a cycle.  
\end{proof}

\begin{prop}\label{critical}
Let $G(V,E)$ be a $(k,r)$-critical digraph then $k(r^2-r+1) \ge |V|$.
\end{prop}

\begin{proof}
We note that by Property $4$, of $G$ being critical, we can deduce that $d^+(v)=r$ for all $v$. Also from the previous proposition any edge needs to have a witness.

We build an undirected graph $H$ with the set of vertices $V$ and with $u,v \in V$ joined by an edge if and only if they have a common parent in $G$. For $v \in V$ we have $d_H(v) \le d^{-}_G(v)(r-1)$ as any potential witness of an edge ending in $v$ must be in $N^{-}(v)$ and each of these $d^{-}_G(v)$ vertices have $r-1$ children apart from $v$.

Turan's Theorem (Theorem $1$, page 95 of \cite{Alon_Spen92}) states that in any undirected graph $K$ there is an independent set of size at least $\sum_{v \in V(K)}{\frac{1}{d_K(v)+1}}.$ Applying this result to $H$ we obtain that there is a set of independent vertices of $H$ with size at least:
$$\sum_{v \in V}{\frac{1}{d_H(v)+1}} \ge \frac{n^2}{\sum_{v \in V}{d_H(v)}+n} \ge \frac{n^2}{(r-1)\sum_{v \in V}{d^{-}_G(v)}+n}=\frac{n^2}{(r-1)rn+n}=\frac{n}{r^2-r+1},$$
where in the first inequality we used Jensen's inequality. Note that an independent set of $m$ vertices in $H$ corresponds to a set of vertices having disjoint in-neighbourhoods so by Proposition \ref{prop:in-nbor} we can find $m$ disjoint cycles in $G$. As $G$ does not contain $k+1$ disjoint cycles we have $m \le k$ so $k(r^2-r+1) \ge |V|$ as desired.
\end{proof}

The following proposition is our main improvement over the result of Alon \cite{Alon96}. We will use it with $t=3$ for which $f(3)=5$ was proved in \cite{k=3}.

\begin{prop}\label{prop:upper_bound_small_general}
Assuming $f(t)=2t-1$ for some $t,$ let $l=\left\lceil \frac{k+1}{t} \right\rceil$, if there exists a $(k,r)$-critical digraph then:
$$  k(r^2-r+1)\left(\sum_{i=0}^{2t-2} \binom{r}{i}\left(\frac{1}{l}\right)^{i}\left(1-\frac{1}{l}\right)^{r-i}\right)+l\left(1-\frac{1}{l}\right)^{r+1} \ge 1.$$
\end{prop}

\begin{proof}
Let us assume the contrary, so $$  k(r^2-r+1)\left(\sum_{i=0}^{2t-2} \binom{r}{i}\left(\frac{1}{l}\right)^{i}\left(1-\frac{1}{l}\right)^{r-i}\right)+l\left(1-\frac{1}{l}\right)^{r+1} < 1.$$
Let $G$ be a $(k,r)$ critical digraph. We colour the vertices of $G$ with $l$ colours, chosen uniformly at random and independently between vertices. Let $A_v$ be the event that less than $2t-1$ children of $v$ are of the same colour as $v$. Let $B_i$ be the event that there are no vertices of colour $i \le l$. We have $\mathbb{P}(A_v)= \sum_{i=0}^{2t-2} \binom{r}{i}\left(\frac{1}{l}\right)^{i}\left(1-\frac{1}{l}\right)^{r-i}$ where term corresponding to $i$ in the sum represents the probability of $v$ having exactly $i$ children of the same colour as it. For each colour $i$ we have $\mathbb{P}(B_i)=\left(1-\frac{1}{l}\right)^n \le \left(1-\frac{1}{l}\right)^{r+1}$. Now using the union bound and $|V|\le k(r^2-r+1)$ given by Proposition \ref{critical} we obtain:
$$ \mathbb{P}\left(\bigcup_{v \in V} A_v \cup \bigcup_{i=1}^{l} B_i\right) \le \sum_{v \in V} \mathbb{P}(A_v)+ \sum_{i=1}^{l} \mathbb{P}(B_i)=$$
$$  |V|\left(\sum_{i=0}^{2t-2} \binom{r}{i}\left(\frac{1}{l}\right)^{i}\left(1-\frac{1}{l}\right)^{r-i}\right)+l\left(1-\frac{1}{l}\right)^{r+1} < $$
$$k(r^2-r+1)\left(\sum_{i=0}^{2t-2} \binom{r}{i}\left(\frac{1}{l}\right)^{i}\left(1-\frac{1}{l}\right)^{r-i}\right)+l\left(1-\frac{1}{l}\right)^{r+1} < 1.$$

So there is a colouring in which vertices of each colour define a non-empty $2t-1$-out digraph which by the assumption of $f(t)=2t-1$ means we find $t$ disjoint cycles in each colour. As there are $l$ colours we get $tl \ge k+1$ disjoint cycles, a contradiction to $G$ being $(k,r)$-critical.
\end{proof}

The following proposition is an approximate version of Proposition \ref{prop:upper_bound_small_general} in the case $t=3.$

\begin{prop} \label{prop:upper_bound_small}
Let $l=\left\lceil \frac{k+1}{3} \right\rceil$, if there exists a $(k,r)$-critical digraph then:
$$ e^{\frac{r-4}{l}} \le \left( \frac{r^6}{8l^3}+\frac{3r^5}{l^2}\right)  .$$
\end{prop}
\begin{proof}
Proposition \ref{prop:upper_bound_small_general} applies for $t=3$ as $f(3)=5$ (proved in \cite{k=3}) giving:

\begin{align*}
1 & \le k(r^2-r+1)\left(\sum_{i=0}^{4} \binom{r}{i}\left(\frac{1}{l}\right)^{i}\left(1-\frac{1}{l}\right)^{r-i}\right)+l\left(1-\frac{1}{l}\right)^{r+1}  \\
& \le 3l(r^2-r+1)\left(\sum_{i=0}^{4}  \frac{r^i}{i!(l-1)^i}\left(1-\frac{1}{l}\right)^{r}\right)+l\left(1-\frac{1}{l}\right)^{r+1}  \\
& \le 3lr^2\left(\sum_{i=0}^{4} \frac{r^i}{(l-1)^ii!}\left(1-\frac{1}{l}\right)^{r}\right)   \\
& = 3lr^2\left(\sum_{i=0}^{3} \frac{1}{i!}\left(\frac{r}{l-1}\right)^i+\frac{r^4}{24(l-1)^4}\right)\left(1-\frac{1}{l}\right)^{r}   \\
& \le 3lr^2\left(\sum_{i=0}^{3} \frac{1}{i!}\frac{1}{4^{3-i}}\left(\frac{r}{l-1}\right)^3+\frac{r^4}{24(l-1)^4}\right)\left(1-\frac{1}{l}\right)^{r}   \\
& = 3lr^2\left(\sum_{i=0}^{3} \frac{4^i}{i!}\left(\frac{r}{4(l-1)}\right)^3+\frac{r^4}{24(l-1)^4}\right)\left(1-\frac{1}{l}\right)^{r}   \\  
& \le 3lr^2\left(\frac{e^4}{64}\left(\frac{r}{l-1}\right)^3+\frac{r^4}{24(l-1)^4}\right)\left(1-\frac{1}{l}\right)^{r}  \\  
& \le \left(\frac{3r^5}{(l-1)^2}+\frac{r^6}{8(l-1)^3}\right)\left(1-\frac{1}{l}\right)^{r-1}  \\
& \le \left(\frac{3r^5}{l^2}+\frac{r^6}{8l^3}\right)\left(1-\frac{1}{l}\right)^{r-4}  \\
& \le \left( \frac{r^6}{8l^3}+\frac{3r^5}{l^2}\right) e^{-\frac{r-4}{l}}.
\end{align*}
where in the 3rd inequality we used $r \ge 2$ and in the 4th inequality we used $1<r/(4(l-1))$ which follows as $r\ge 2k+1 > 4(l-1).$
\end{proof}

\begin{cor}
$f(k)$ is finite for all values of $k$.
\end{cor}

\begin{proof}
Let us assume the opposite and let $k$ be the smallest integer such that $f(k+1)$ is not finite. Hence for any $r$ this implies that there is an $r$-out digraph not containing $k+1$ cycles, taking one with minimal number of vertices and among such graphs minimal number of edges gives a graph satisfying properties $1)-4)$ of $(k,r)$-criticallity. Assuming $r\ge f(k)+2,$ gives that any graph on $r-2$ vertices contains $k$ disjoint cycles giving property $5)$. This shows that for any $r\ge f(k)+2,$ there is an $(k,r)$-critical digraph. This contradicts Proposition \ref{prop:upper_bound_small} as there the $LHS$ is exponential in $r$ while the $RHS$ is a polynomial. \end{proof}

A more careful bounding in the above derivation would in fact give us $f(k+1) \le (1+o(1))k\log_{e}(k)$.  

The following proposition is slightly technical as we try to stay very close to equality in various inequalities we use, which makes the calculations slightly awkward. 

\begin{prop}\label{prop:small-cases}
For $k < 2^{13}$ we have $f(k) \le 15k+30$ and for $k\le 10$ we have $f(k)\le 12k$.
\end{prop}

\begin{proof}
Let $k$ be the smallest integer for which $f(k+1)>15(k+1)+30$. Let $r=f(k+1)-1,$ as $f(k)\le 15k+30$ we have that $r-2 \ge f(k)$ so there exists a $(k,r)$ critical digraph and Proposition \ref{prop:upper_bound_small} applies. Denoting $x=r/l$ it implies $l^3e^{\frac{4}{l}} \ge \frac{e^x}{x^6/8+3x^5}$. The right hand side is an increasing function on $x \ge 4$ and the left hand side is increasing in $l$ for $l \ge 2$. If $x \ge 45$ by using $l\le \frac{k}{3}+1 \le 2^{13}/3+1$ we get $21 \cdot 10^{9}  > l^3e^{\frac{4}{l}} \ge \frac{e^{45}}{45^6/8+3\cdot45^5} > 21 \cdot 10^{9},$ so a contradiction. Hence, $f(k+1)-1 < 45l=45 \left\lceil \frac{k+1}{3} \right\rceil  \le 15(k+1)+30,$ also a contradiction completing the proof of the first claim.

For the second claim, the results holds when $k\le 3$ as $f(k)=2k-1<=12k$ in this case, so we assume $k \ge 4.$ We repeat the argument with $12k$ in place of $15k+30$, now $k\le 10$ implies $l < 4,$ so if $x \ge 24$ we obtain $200>l^3e^{\frac{4}{l}} \ge \frac{e^x}{x^6/8+3x^5}>200,$ so a contradiction. Hence, $f(k+1)-1 < 24l=24 \left\lceil \frac{k+1}{3} \right\rceil  \le 8(k+1)+16\le 12k.$ Completing the proof.
\end{proof}

The main idea which enables us to find a linear bound, due to Alon \cite{Alon96}, is based on an extension of the probabilistic ideas introduced in the proof of Proposition \ref{prop:upper_bound_small} and it relies on the following bound. 

\begin{prop}
Assuming $r \ge 2^{15}$ one can split an $(k,r)$-critical digraph $G$ into $2$ disjoint $\frac{r}{2}-\frac{r^{2/3}}{ \sqrt{2}}$-out digraphs.
\end{prop}

\begin{proof}
We colour the vertices of $G$ into $2$ colours uniformly at random. Let $A_v$ be the event that $v$ has less than $t=\frac{r}{2}-\frac{r^{2/3}}{\sqrt{2}}$ children of the same colour. Let $X_v$ be a random variable counting the number of children of $v$ having the same colour as $v$. We have $X_v \sim \text{Bin}(r, 1/2)$ and $A_v$ is equivalent to $X_v<t$. Now using the bound of Chernoff for the tail of the binomial random variable, given in \cite{Alon_Spen92} Appendix A, we have $\mathbb{P}(A_v) \le e^{-r^{1/3}}$.

Now, $r \ge 2k+1$ and Proposition \ref{critical} give us $|V|<\frac{r^3}{2}$, and $G$ being $r$-out implies $|V| \ge r+1,$ these inequalities imply $2\cdot 2^{-|V|}+\frac{r^3}{2}e^{-r^{1/3}}<1,$ using $r \ge 2^{15},$ so there is a colouring of the vertices such that each colour occurs and represents a disjoint $t$-out subgraph.
\end{proof}

\begin{cor}\label{cor:recursion}
Given $f(h)>2^{15}$ we have:
$$f(h)-\sqrt{2}(f(h))^{2/3} \le 2f\left(\lceil h/2\rceil\right) \text{ or}$$
$$f(h) \le f(h-1)+2$$
\end{cor}

\begin{proof}
Let us assume the second inequality fails, so $f(h)>f(h-1)+2.$ Let $h=k+1$, $r=f(k+1)-1 \ge f(k)+2,$ so that any $r-2$-out digraph contains $k$ disjoint cycles and there is an $r$-out digraph which does not contain $k+1$ disjoint cycles, implying there exists an $(k,r)$-critical graph $G$. By the previous proposition $G$ contains $2$ disjoint $t=\left \lceil r/2-r^{2/3}/ \sqrt{2}\right \rceil$-out subgraphs. We must have $t<f\left(\left\lceil\frac{k+1}{2}\right\rceil\right)$ as else we can find $k+1$ disjoint cycles. Hence,
$$\frac{f(h)-1}{2}-\frac{(f(h))^{2/3}}{\sqrt{2}} \le  \frac{f(h)-1}{2}-\frac{(f(h)-1)^{2/3}}{\sqrt{2}}=\frac{r}{2}-\frac{r^{2/3}}{\sqrt{2}} \le f\left(\lceil h/2\rceil\right)-1.$$
Which implies the desired result.
\end{proof}

\begin{lem}\label{lem:big-cases}
For $h \ge 2^{12}$ we have $f(h) \le 18h-45h^{2/3}$.
\end{lem}

\begin{proof}
Notice that for $2^{12} \le h < 2^{13}$ the result holds as a consequence of Proposition \ref{prop:small-cases} as in this range:
$$f(h) \le 15h+30 \le 18h-45h^{2/3}.$$
Let us assume for the sake of contradiction that the Lemma fails and Let $h$ be the smallest integer for which the Lemma fails, so 
\begin{equation}\label{equ:1}
f(h)>18h-45h^{2/3}
\end{equation}
and for all $t<h$ we have
\begin{equation}\label{equ:2}
f(t) \le 18t-45t^{2/3}.
\end{equation}
By above consideration $h \ge 2^{13},$ so if $f(h) \le 2^{15}$ then $f(h) \le 4h \le 18h-45h^{2/3}.$ Hence, we may and will assume that $f(h) > 2^{15}$. This allows us to apply Corollary \ref{cor:recursion} for $f(h)$ implying that $f(h)-\sqrt{2}f(h)^{2/3} \ge f\left(\lceil h/2 \rceil\right)$ or $f(h) \le f(h-1)+2.$ In the latter case 
\begin{align*}
f(h) & \le f(h-1)+2 \\
& \le 18(h-1)-45(h-1)^{2/3}+2 \\
& = 18h-45h^{2/3}+45(h^{2/3}-(h-1)^{2/3})-16 \\
& < 18h - 45h^{2/3},
\end{align*}
where we used (\ref{equ:2}) for $t=h-1<h$ and $h^{2/3}-(h-1)^{2/3} < 1/3,$ which holds for $h \ge 2^5.$ This gives a contradiction.

If the former case  we have  
\begin{align*}
f(h)-\sqrt{2}f(h)^{2/3} & \le 2f\left(\lceil h/2\rceil\right) \\
 & \le 2\left(18\lceil h/2\rceil-45 \left(\lceil h/2\rceil\right)^{2/3}\right) \\
 & \le 2\left(18(h+1)/2-45 ( (h+1)/2)^{2/3}\right)\\
 & =18h+18-2^{1/3}45(h+1)^{2/3}.
\end{align*}
Where we used (\ref{equ:2}) for $t=\lceil h/2\rceil$, and the fact $g(x)=18x-45x^{2/3}$ is increasing for $x \ge 5.$

Using that $g(x)=x-\sqrt{2}x^{2/3}$ is increasing for $x \ge 2^{3/2} $ and (\ref{equ:1}) we obtain: 
\begin{align*}
f(h)-\sqrt{2}f(h)^{2/3} & \ge 18h-45h^{2/3} - \sqrt{2} (18h-45h^{2/3})^{2/3} \\
& \ge 18h-45h^{2/3}- \sqrt{2} (18h)^{2/3},
\end{align*}
implying in turn $18+\sqrt{2}\cdot(18h)^{2/3} > (2^{1/3}-1)45h^{2/3}.$ Using $h \ge 2^{12}$ we get $1+10 > 18h^{-2/3}+\sqrt{2}\cdot18^{2/3} > (2^{1/3}-1)\cdot 45>11,$ where we used $10>\sqrt{2}\cdot18^{2/3}$ and $2^{1/3}>5/4,$ giving us a contradiction and completing the proof.
\end{proof}
\section{Concluding remarks}
Combining results of Propositions \ref{prop:small-cases} and Lemma \ref{lem:big-cases} we obtain the following result.
\begin{thm}
$$2k-1 \le f(k) \le 18k.$$
\end{thm}
\vspace{-2mm}
Optimising various constants and inequalities used above and using a computer together with Proposition \ref{prop:upper_bound_small_general} in place of Proposition \ref{prop:small-cases} the constant can be improved to close to $16$ but we believe that without introducing some new ideas, the exact method given above is unlikely to bring the constant much lower than $16$.

The main benefit of Proposition \ref{prop:upper_bound_small_general} is that it allows further improvement of the constant, provided we could prove $f(k)=2k-1$ for more small values of $k$. Unfortunately, this is still an open problem for all $k \ge 4$. We do however note that Proposition \ref{critical} shows that if for a fixed $k$ the conjecture $f(k)=2k-1$ is false then there needs to exist a counterexample having at most $4k^3$ vertices, in particular testing the conjecture for small values of $k$ can be done in a finite time, albeit still prohibitively large, even for $k=4$. 

\vspace{-2mm}

\section*{Acknowledgments}
I am very grateful to Prof. Imre Leader for his support and useful discussions. 

\noindent I was provided financial support by Trinity College, Cambridge for which I am also grateful.

I owe a great debt to the anonymous referees, whose excellent and insightful comments vastly improved presentation of the paper. Specifically, a referee brought to my attention the reference \cite{k=3}, using which allowed for an improvement to the main result.

\end{document}